\documentclass[11pt,a4paper]{amsart}

\usepackage{amsmath,amssymb,amsthm,amscd}
\usepackage[all]{xy}

\title[Long exact sequences coming from the snake lemma]
{A characterization of long exact sequences coming from the snake lemma}
\date{\today}

\author{Jan \v S\v tov\'\i\v cek}
\address{
Institutt for matematiske fag \\
NTNU \\
N--7491 Trondheim, Norway
}
\email{stovicek@math.ntnu.no}

\thanks{The author was supported by the Research Council of Norway through the Storforsk-project ``Homological and geometric methods in algebra''}


\renewcommand{\iff}{if and only if }
\newcommand{\st}{such that }
\DeclareMathOperator{\Hom}{Hom}
\DeclareMathOperator{\stHom}{\underline{Hom}}

\DeclareMathOperator{\Ext}{Ext}

\DeclareMathOperator{\Ker}{Ker}

\DeclareMathOperator{\Coker}{Coker}

\newcommand{\trl}[1]{^{\le {#1}}}
\newcommand{\trr}[1]{^{\ge {#1}}}
\newcommand{\Toda}[3]{\langle {#1},{#2},{#3} \rangle}
\newcommand{\A}{\mathcal{A}}

\newcommand{\D}{\mathcal{D}}

\newcommand{\T}{\mathcal{T}}
\newcommand{\Der}[1]{\mathbf{D}^b({#1})}

\newcommand{\modR}{\hbox{\rm mod-}R}
\newcommand{\perfR}{\hbox{\rm perf-}R}
\newcommand{\stmodR}{\hbox{\rm \underline{mod}-}R}
\theoremstyle{plain}
\newtheorem{thm}{Theorem}
\newtheorem{prop}[thm]{Proposition}

\newtheorem*{question}{Question}
\theoremstyle{definition}
\newtheorem{defn}[thm]{Definition}
\theoremstyle{remark}
\newtheorem{rem}[thm]{Remark}
\newtheorem{expl}[thm]{Example}

\begin{document}
\begin{abstract}
Given an abelian category, we characterize the long exact sequences of
length six which can be obtained from the snake lemma. Equivalently,
these are the long exact sequences which arise as the homology of a
triangle in the corresponding derived bounded category.
\end{abstract}

\maketitle

% -----------------------------------------------------------------------------
\section{Introduction and preliminaries}

Let $\A$ be an abelian category. We aim to answer the following
question (see also the question asked by Deligne at the end of the first part
of~\cite{N}):
\begin{question}
Given a long exact sequence
$$
0 \to A \overset{a}\longrightarrow
B \overset{b}\longrightarrow
C \overset{c}\longrightarrow
D \overset{d}\longrightarrow
E \overset{e}\longrightarrow
F \to 0
$$
in $\A$, which conditions must it satisfy so that we can obtain it
from the snake lemma as
$$
0 \to \Ker f_1 \to  \Ker f_2 \to  \Ker f_3 \to
\Coker f_1 \to \Coker f_2 \to \Coker f_3 \to 0,
$$
where $f = (f_1,f_2,f_3)$ is a suitable homomorphism between short
exact sequences in $\A$?
\end{question}

This question is an instance of a more general problem. Let $\Der\A$
be the derived bounded category of $\A$. We ignore the possible
set-theoretic difficulties mentioned in~\cite[\S 2.2]{N2}, because in
what we intend to do with $\Der\A$ here they do not play any role.
Now, one can ask when a long exact sequence
$$
\epsilon: \quad 0 \to A_1 \longrightarrow A_2 \longrightarrow \dots \longrightarrow A_n \to 0
$$
in $\A$ arises as the homology sequence of a triangle $X \to Y \to Z \to X[1]$
in $\Der\A$.

For $n = 5$ this has been completely resolved by Neeman~\cite{N}: the long
exact sequence $\epsilon$ is the homology of a triangle
\iff the corresponding class in $\Ext^3_\A(A_5,A_1)$ vanishes. It is worth to
mention that Neeman proved more. Namely, if $\epsilon$ comes from a triangle for $n>5$,
then the corresponding class in $\Ext^3_\A$ vanishes for every bit of the form
$$
0 \to K \longrightarrow A_{i-1} \longrightarrow A_{i} \longrightarrow A_{i+1} \longrightarrow C \to 0.
$$

One may wonder what exactly happens for $n=6$. We have a necessary condition for two classes in $\Ext^3_\A$, but this is not sufficient as we illustrate in an example at the end of this note. It turns out that we must impose one more condition on a so called Toda bracket. This way, we answer Deligne's question and give a complete solution for $n = 6$. Looking back at the very first question about the snake lemma, it is straightforward to see that a long exact sequence of length six comes from the snake lemma \iff it arises as the homology of a triangle.

\smallskip

In fact, we will tackle a yet more general problem---here we follow
the formalism from~\cite{N} again. Namely, it is assumed
that $\A$ arises as the heart of a t-structure $(\D\trl0, \D\trr0)$ in
a triangulated category $\T$. In such a case we have a homological
functor $H: \T \to \A$ sending $X \in \T$ to $(X\trl0)\trr0$, and we
let $H^n(X) = H(X[n])$. We refer to~\cite[\S1.3]{BBD} for details.

It often happens that there is an exact functor $G: \Der\A \to \T$
extending the identity on $\A$. Then necessarily
$$ \Hom_{\Der\A}(X,Y[i]) \to \Hom_\T(GX,GY[i]) $$
is an isomorphism for $X,Y \in \A$ and $i = 0,1$. For $i \ge 2$,
however, the canonical morphism between the homomorphism groups is far
from being an isomorphism in general.

If there is such a functor $G$, we study the problem when a long exact
sequence in $\A$ is realized as the homology of a triangle
$X \to Y \to Z \to X[1]$ in $\T$ via the homological functor
$H: \T \to \A$. The original problem can be reconstructed by taking
$\T = \Der\A$ together with the canonical t-structure
$(\D\trl0, \D\trr0)$, and putting $G = \mathbf{1}_{\Der\A}$.

\smallskip

For formulating the result, we will need the concept of a Toda
bracket. It was first used in~\cite{T} and the definition was
extended in~\cite{C}, in both cases in a topological context. Here, we
give a purely algebraic definition:

\begin{defn} \label{defn:toda}
Let $\T$ be a triangulated category and
$$
\begin{CD}
X @>x>> Y @>y>> Z @>z>> W
\end{CD}
$$
be three composable morphisms in $\T$. Consider the following triangle
in $\T$:
$$
\begin{CD}
Z @>>> V @>>> Y[1] @>{y[1]}>> Z[1]
\end{CD}
$$
Then the \emph{Toda bracket} $\Toda zyx$ is defined as the set of all
morphisms $X[1] \to W$ which can be expressed as the composition of
some $g: X[1] \to V$ and $f: V \to W$ making the following diagram
commutative in $\T$:
$$
\xymatrix{
  && X[1] \ar[d]^{x[1]} \ar@{.>}[dl]_g \\
  Z \ar[r] \ar[d]_z & V \ar[r] \ar@{.>}[dl]^f & Y[1] \ar[r]^{y[1]} & Z[1] \\
  W
}
$$
\end{defn}

It not difficult to see from the definition that $\Toda zyx$ is either empty or a
coset of the subgroup $z\Hom(X[1],Z)+\Hom(Y[1],W)x[1]$ of $\Hom(X[1],W)$.
Moreover, $\Toda zyx \ne \emptyset$ \iff
$$
\begin{CD}
X @>x>> Y @>y>> Z @>z>> W
\end{CD}
$$
is a chain complex, that is $zy = 0 = yx$.

We will mostly be interested in whether $0 \in \Toda zyx$. This is by
far not automatically satisfied even if $\Toda zyx$ is non-empty. A
simple way of generating examples of this phenomenon is described
in~\cite[\S 3, pg. 219]{BN}; we will use this idea in the example at
the end of this note.

% -----------------------------------------------------------------------------
\section{The main result}

Our goal here is to prove the following statement:

\begin{thm} \label{thm:main}
Let $\A$ be the heart of a t-structure in a triangulated category $\T$
and assume that there is an exact functor $G: \Der\A \to \T$ extending the identity on $\A$.
Let further
$$
0 \to A \overset{a}\longrightarrow
B \overset{b}\longrightarrow
C \overset{c}\longrightarrow
D \overset{d}\longrightarrow
E \overset{e}\longrightarrow
F \to 0 \eqno{(*)}
$$
be an exact sequence in $\A$. Denote by $K, L, M$ be the images of
$B \to C$, $C \to D$ and $D \to E$, respectively, and denote as
follows the morphisms in $\T$ coming from the short exact subsequences
of $(*)$:
$$
\begin{CD}
F @>{\alpha}>> M[1] @>{\beta}>> L[2] @>{\gamma}>> K[3] @>{\delta}>> A[4].
\end{CD}
$$
Then the following are equivalent:
\begin{enumerate}
\item There exist a triangle $X \to Y \to Z \to X[1]$ in $\T$ whose
homology is isomorphic to $(*)$.
\item $0 \in \Toda \delta{\gamma\beta} \alpha$.
\end{enumerate}
\end{thm}

\begin{rem} \label{rem:interpret}
Note that $0 \in \Toda \delta{\gamma\beta}\alpha$ automatically
implies that $\delta\gamma\beta = 0 = \gamma\beta\alpha$. That is, the
image under $G: \Der\A \to \T$ of the corresponding classes in
$\Ext^3_\A(M,A)$ and $\Ext^3(F,K)$ coming from $(*)$ vanishes.
% The necessity of this condition was proved in~\cite{N}. We will show in
% the next section that it is not sufficient, though.

If we already know that $\delta\gamma\beta = 0 = \gamma\beta\alpha$,
then the condition $0 \in \Toda \delta{\gamma\beta}\alpha$ can be
restated as follows. Let $W$ fit into the triangle
$D[-1] \to W \to C \overset{c}\to D$. Then there is a natural triangle
$$
\begin{CD}
K[2] @>>> W[2] @>>> M[1] @>{\gamma\beta}>> K[3]
\end{CD}
$$
coming from the exact sequence $0 \to K \longrightarrow C \overset{c}\longrightarrow D \longrightarrow M \to 0$.
Moreover, we know that the morphisms $F \to M[1]$ and
$K[2] \to A[3]$ coming from $(*)$ factorize (non-uniquely) through
$W[2] \to M[1]$ and $K[2] \to W[2]$, respectively:
$$
\xymatrix{
  && F \ar[d]^\alpha \ar@{.>}[dl] \\
  K[2] \ar[r] \ar[d]_{\delta[-1]} & W[2] \ar[r] \ar@{.>}[dl] & M[1] \ar[r]^{\gamma\beta} & K[3] \\
  A[3]
}
$$
The assertion in Theorem~\ref{thm:main}(2) is then equivalent to being able to choose
the factorization so that the composition $F \to W[2] \to A[3]$ vanishes.
\end{rem}

\begin{proof}
$(1) \implies (2)$. Assume there is a triangle
$$
\begin{CD}
X @>x>> Y @>y>> Z @>z>> X[1]
\end{CD}
$$
in $\T$ whose homology is $(*)$; say $H^0(X) = A$. Then the main
result of~\cite{N} says that
$\delta\gamma\beta = 0 = \gamma\beta\alpha$ and hence
$\Toda zyx \ne \emptyset$, but this will also follow from our argument
directly.

We need to prove more, namely that $0 \in \Toda zyx$. To do this, note first that we can without loss of generality assume that $X,Y,Z \in \D\trr{-1} \cap \D\trl1$. If not we can replace the original triangle $X \to Y \to Z \to X[1]$ by:
$$
\begin{CD}
(X\trl1)\trr0 @>{(x\trl1)\trr0}>> (Y\trl1)\trr0 @>{y'}>> Z' @>{z'}>> (X\trl1)\trr0[1].
\end{CD}
$$
The long exact sequence of homologies remains unchanged by~\cite[Proposition 1.3.5]{BBD}, and we have $(X\trl1)\trr0, (Y\trl1)\trr0, Z' \in \D\trr{-1} \cap \D\trl1$.

Let us next consider the truncations of $X$ with respect to
$(\D\trl0, \D\trr0)$:
$$
\begin{CD}
X\trl0 @>>> X @>>> X\trr1 @>>> (X\trl0)[1]
\end{CD}
$$
An easy computation of homologies shows that $X\trl0 \cong A$ and
$X \trr1 \cong D[-1]$. Note that here we need the fact that $X \in \D\trr{-1} \cap \D\trl1$, since in general it is perfectly possible that there are non-zero objects in $\T$ whose homologies with respect to $(\D\trl0,\D\trr0)$ all vanish. By forming a homotopy push-out of
$$
\begin{CD}
X @>x>> Y \\
@VVV      \\
X\trr1
\end{CD}
$$
we subsequently obtain the following diagram with triangles in rows and columns:
$$
\begin{CD}
A     @=     A                          \\
@VVV         @VVV                       \\
X     @>x>>  Y    @>y>>  Z  @>z>> X[1]  \\
@VVV         @VtVV       @|       @VVV  \\
D[-1] @>r>>  V    @>v>>  Z  @>>>  D     \\
@VVV         @VhVV                      \\
A[1]  @=     A[1]
\end{CD}
\eqno{(\dag)}
$$
Similarly, one easily computes that $Z\trl0 \cong C$ and
$Z\trr1 \cong F[-1]$. The truncation with respect to the t-structure
gives a triangle
$$
\begin{CD}
F[-2] @>>> C @>>> Z @>>> F[-1],
\end{CD}
$$
and by constructing a homotopy pull-back of
$$
\begin{CD}
  @.    C  \\
@.   @VVV  \\
V @>v>> Z,
\end{CD}
$$
one obtains a diagram:
$$
\begin{CD}
      @.    F[-2] @=     F[-2]         \\
@.          @VfVV        @VVV          \\
D[-1] @>>>  W     @>s>>  C    @>c>> D  \\
@|          @VgVV        @VVV       @| \\
D[-1] @>r>> V     @>v>>  Z    @>>>  D  \\
@.          @VVV         @VVV          \\
      @.    F[-1] @=     F[-1].
\end{CD}
\eqno{(\ddag)}
$$
Clearly, the composition
$$
\begin{CD}
F @>{f[2]}>> W[2] @>{(hg)[2]}>> A[3],
\end{CD}
$$
is zero, where the morphisms $f$, $g$ and $h$ come from diagrams
$(\dag)$ and $(\ddag)$.

All we have to show now is that the morphisms
$f[2]: F \to W[2]$ and $(hg)[2]: W[2] \to A[3]$ fit up to isomorphism the definition of
the Toda bracket. More precisely, consider
the triangle
$$
\begin{CD}
  K[2] @>{\epsilon}>> W[2] @>{\zeta}>> M[1] @>{\gamma\beta}>> K[3].
\end{CD}
$$
We will prove that there are automorphisms $\phi$ and $\psi$ of $F$ and $A[3]$, respectively, \st the following equalities hold:
\begin{align*}
\alpha &= \zeta \circ f[2] \circ \phi, \\
\delta[-1] &= \psi \circ (hg)[2] \circ \epsilon.
\end{align*}

For the first equality, note that $W\trl0 \cong K \cong V\trl0$ and
$g\trl0: W\trl0 \to V\trl0$ is an isomorphism. Taking into account
that $W\trr1 \cong M[-1]$ and $V\trr1 \cong E[-1]$, we get the
following commutative diagram with triangles in rows:
$$
\begin{CD}
K  @>{\epsilon[-2]}>> W     @>{\zeta[-2]}>> M[-1] @>>> K[1]  \\
@|                  @VgVV         @.                   @|    \\
K       @>>>          V          @>>>       E[-1] @>>> K[1].
\end{CD}
$$

We can choose the morphisms in the upper row as stated since the triangle there is a truncation triangle of $W$ with respect to the t-structure and as such it is unique up to a unique isomorphism. By~\cite[Lemma 1.4.3]{N2}, we can complete this diagram with a morphism $M[-1] \to E[-1]$ \st the resulting morphism of triangles can be completed to an octahedron. Note that there is in fact only one morphism $M[-1] \to E[-1]$ which completes the diagram, namely $g\trr1$, and it is up to isomorphism just the shift of the monomorphism $M \to E$ in $\A$ coming from $(*)$ in the statement of Theorem~\ref{thm:main}.
We can summarize our findings in the following diagram with triangles in rows and columns:
$$
\begin{CD}
   @.               F[-2]   @>{\phi'}>>     F[-2]            \\
@.                  @VfVV              @V{\alpha[-2]}VV      \\
K  @>{\epsilon[-2]}>> W     @>{\zeta[-2]}>> M[-1] @>>> K[1]  \\
@|                  @VgVV                   @VVV       @|    \\
K  @>k>>              V     @>>>            E[-1] @>>> K[1]  \\
@.                  @VVV                 @V{e[-1]}VV         \\
   @.               F[-1]   @>{\phi'[1]}>>  F[-1]            \\
\end{CD}
\eqno{(\Delta)}
$$
%
%$$
%\begin{CD}
%      @.               K     @=           K                       \\
%@.                     @V{\epsilon[-2]}VV @VkVV                   \\
%F[-2] @>f>>            W     @>g>>        V     @>>>        F[-1] \\
%@|                     @V{\zeta[-2]}VV    @VVV              @|    \\
%F[-2] @>{\alpha[-2]}>> M[-1] @>>>         E[-1] @>{e[-1]}>> F[-1] \\
%@.                     @VVV               @VVV                    \\
%      @.               K[1]  @=           K[1]                    \\
%\end{CD}
%\eqno{(\Delta)}
%$$
%
We do not know whether $\phi'$ is the identity in general, but it certainly is an isomorphism since $(\Delta)$ is an instance of the octahedral axiom. The desired equality appears in the top square in the diagram when taking $\phi = (\phi'[2])^{-1}$.
 
Let us finally prove that $\delta[-1] = \psi \circ (hg)[2] \circ \epsilon$ for some automorphism $\psi$. As before, one readily checks that
$Y \trr1 \cong E[-1] \cong V\trr1$ and
$t\trr1: Y \trr1 \to V\trr1$ is an isomorphism (see diagram $(\dag)$
for the morphism $t: Y \to V$), and that one can obtain a diagram
$$
\begin{CD}
  E[-2] @>>> B    @>>>  Y     @>>> E[-1]  \\
  @|         @VVV      @VtVV      @|      \\
  E[-2] @>>> K    @>k>> V     @>>> E[-1]
\end{CD}
$$
which admits a completion to an octahedron. Again, the morphism $B \to K$ in the diagram is up to isomorphism nothing else than the
epimorphism in $\A$ coming from $(*)$ in the statement of
Theorem~\ref{thm:main}. Therefore, we obtain the following commutative
diagram with triangles in rows and columns:
$$
\begin{CD}
      @.     A    @>{\psi'[-1]}>> A                 \\
@.         @VaVV                @VVV                \\
E[-2] @>>>   B    @>>>            Y     @>>> E[-1]  \\
@|         @VVV                 @VtVV      @|       \\
E[-2] @>>>   K    @>k>>           V     @>>> E[-1]  \\
@.     @V{\delta[-3]}VV         @VhVV               \\
      @.   A[1]   @>{\psi'}>>    A[1]               \\
\end{CD}
\eqno{(\Xi)}
$$
%
%$$
%\begin{CD}
%      @.    E[-2] @=           E[-2]                       \\
%@.          @VVV               @VVV                        \\
%A     @>a>> B     @>>>         K     @>{\delta[-3]}>> A[1] \\
%@|          @VVV               @VkVV                  @|   \\
%A     @>>>  Y     @>t>>        V     @>h>>            A[1] \\
%@.          @VVV               @VVV                        \\
%      @.    E[-1] @=           E[-1]                       \\
%\end{CD}
%\eqno{(\Xi)}
%$$

Since $(\Xi)$ is an instance of the octahedral axiom, $\psi'$ is an isomorphism.
From the bottom square we see that $\delta[-3] = (\psi')^{-1} hk$, and by combining
this with the equality coming from the leftmost square of diagram
$(\Delta)$, we get $\delta[-3] = (\psi')^{-1} hg \circ \epsilon[-2]$. The equality we
wanted to prove is obtained by shifting this by two and putting $\psi = (\psi'[2])^{-1}$. This concludes
the proof of $(1) \implies (2)$.

\smallskip

$(2) \implies (1)$. This follows basically by retracing the
steps. Assume we have an exact sequence as in $(*)$ satisfying
$(2)$. Let us denote the morphisms constructed in
Remark~\ref{rem:interpret} by $f': F \to W[2]$ and
$g': W[2] \to A[3]$. First we put $f = f'[-2]$ and form the triangles
$$
\begin{CD}
  D[-1] @>>> W @>s>> C @>c>> D
\end{CD}
\eqno{(\sharp)}
$$
and
$$
\begin{CD}
  F[-2] @>f>> W @>g>> V @>>> F[-1].
\end{CD}
$$
Next we form diagram $(\ddag)$ as a homotopy pushout of
$$
\begin{CD}
W      @>s>>   C   \\
@VgVV              \\
V
\end{CD}
$$

Because $g'[-2] \circ f = 0$ by assumption, there is a morphism
$h: V \to A[1]$ \st $hg = g'[-2]$. Considering the triangle
$$
\begin{CD}
A @>>> Y @>t>> V @>h>> A[1],
\end{CD}
$$
this allows us to form diagram $(\dag)$ as a homotopy pull-back of
$$
\begin{CD}
      @.    Y   \\
@.        @VtVV \\
D[-1] @>r>> V.
\end{CD}
$$
Now we can define the triangle
$$
\begin{CD}
X @>x>> Y @>y>> Z @>z>> X[1]
\end{CD}
\eqno{(\diamondsuit)}
$$
for condition $(1)$ as the second row of $(\dag)$ and we must prove
that its homology is indeed isomorphic to $(*)$.

Also here we only reverse the arguments. We know that the homology of
triangle $(\sharp)$ is isomorphic to
$$
0 \to K \longrightarrow C \overset{c}\longrightarrow D \longrightarrow M \to 0.
$$
By examining columns in diagram $(\Delta)$, which in this case can be constructed with $\phi' = \mathbf{1}_{F[-2]}$, starting with the commutative square
$$
\begin{CD}
F[-2]  @=                F[-2]\,           \\
@VfVV                  @V{\alpha[-2]}VV    \\
W      @>{\zeta[-2]}>>   M[-1],
\end{CD}
$$

one immediately sees that
$$
0 \to H^1(W) \overset{H^1(g)}\longrightarrow H^1(V) \longrightarrow H^1(F[-1]) \to 0
$$
is a short exact sequence corresponding to $\alpha: F \to M[1]$. Then
it easily follows from diagram $(\ddag)$ that the homology of the
triangle $D[-1] \overset{r}\to V \overset{v}\to Z \to D$ is precisely
$$
0 \to K \longrightarrow C \overset{c}\longrightarrow D \overset{d}\longrightarrow E \overset{e}\longrightarrow F \to 0.
$$
A similar argument shows that the homology of triangle $(\diamondsuit)$
is the long exact sequence $(*)$ from the statement of the theorem,
which concludes the proof.
\end{proof}

% -----------------------------------------------------------------------------
\section{An example}

It is relatively easy to construct an exact sequence of length six
which does not come from the snake lemma. Namely, let $R$ be any ring
admitting a module $M$ of projective dimension at least $4$ and
consider the beginning of the projective resolution
$$
0 \to \Omega^4 M \longrightarrow P_3 \longrightarrow P_2 \longrightarrow P_1 \longrightarrow P_0 \longrightarrow M \to 0.
$$
Such a sequence certainly cannot come from the snake lemma because it
has a non-zero class in $\Ext_R^4(M,\Omega^4 M)$.

We are going to construct an example with finer properties. Namely, if
$k$ is a field and $R = k[x]/(x^3)$, we will construct an exact sequence
$$
0 \to A \overset{a}\longrightarrow
B \overset{b}\longrightarrow
C \overset{c}\longrightarrow
D \overset{d}\longrightarrow
E \overset{e}\longrightarrow
F \to 0
$$
\st in the notation of Theorem~\ref{thm:main}:
\begin{enumerate}
\item $\delta\gamma\beta = 0 = \gamma\beta\alpha$,
\item $0 \not\in \Toda \delta{\gamma\beta} \alpha$.
\end{enumerate}
This will also show that the necessary condition given by
Neeman~\cite{N} is not sufficient.

In order to verify the properties of our forthcoming example, we need a little more
theory. If $R$ is a finite dimensional self-injective algebra over a
field $k$, we denote by $\modR$ the category of all finite dimensional
modules and by $\stmodR$ the corresponding stable category modulo
projectives-injectives. It is well known that $\stmodR$ has a natural
triangulated structure with the shift functor being the cosyzygy
functor $\Omega^-$ (see for example~\cite[\S 1]{H}). Moreover, there
is a close link between $\Der\modR$ and $\stmodR$:

\begin{prop} \label{prop:stmod}
Let $R$ be a finite dimensional self-injective algebra over a
field. Then there is an exact functor $Q: \Der\modR \to \stmodR$ \st
\begin{enumerate}
\item The following diagram with the obvious functors is commutative:
$$
\xymatrix{
\Der\modR \ar[r]^(.57)Q & \stmodR  \\
\modR \ar[u] \ar[ur]
}
$$
\item The induced homomorphisms
$$ \Hom_{\Der\modR}(X,Y[i]) \longrightarrow \stHom(QX,\Omega^{-i} QY) $$
are isomorphisms for each $X,Y \in \modR$ and $i \ge 1$.
\end{enumerate}
\end{prop}

\begin{proof}
$(1)$. This follows from (the proof of) \cite[Theorem 2.1]{R}. Namely, if
$\perfR$ stands for the full triangulated subcategory of $\Der\modR$
formed by the perfect complexes, then $Q$ is obtained as a composition
of the localization functor
$$ Q': \Der\modR \longrightarrow \Der\modR / \perfR $$
with a quasi-inverse to the natural functor (which is a triangle
equivalence by~\cite{R}):
$$ \stmodR \longrightarrow \Der\modR / \perfR. $$

$(2)$. This follows by considering the isomorphisms
$$
\Hom_{\Der\modR}(X,Y[i]) \cong
\Ext^i_R(X,Y) \cong
\stHom(X,\Omega^{-i} Y)
$$
and taking into account the construction of $Q$.
\end{proof}

Now we can construct the example.

\begin{expl}
Let $k$ be a field and $R = k[x]/(x^3)$. Let
$Q: \Der\modR \to \stmodR$ be a functor as in
Proposition~\ref{prop:stmod}. Let us further consider the unique composition series
$$
0 \subseteq S \subseteq N \subseteq R
$$
of $R$. It is well known that $S$, $N$ and $R$ are up to isomorphism
the only indecomposables in $\modR$. The non-split exact sequence
$0 \to S \longrightarrow N \longrightarrow S \to 0$ in $\modR$ yields the following triangle
in $\stmodR$:
$$
\begin{CD}
S @>>> N @>>> S @>>> N (= \Omega^- S).
\end{CD}
$$
This can be rewritten as:
$$
\begin{CD}
  S @>{\alpha'}>>
  \Omega^- S @>{\eta'}>>
  \Omega^{-3} N @>{\delta'}>>
  \Omega^{-4} N
\end{CD}
$$
Obviously $\delta'\eta' = 0 = \eta'\alpha'$ and one easily checks that $0 \not\in \Toda {\delta'}{\eta'}{\alpha'}$, for example by using the comment on contractible triangles from~\cite[\S 3, pg. 219]{BN}. By Proposition~\ref{prop:stmod}, we can find morphisms
$$
\begin{CD}
  S @>{\alpha}>>
  S[1] @>{\eta}>>
  N[3] @>{\delta}>>
  N[4]
\end{CD}
$$
in $\Der\modR$ \st the images of $\alpha,\eta,\delta$ under $Q$ are
$\alpha',\eta',\delta'$, respectively. It follows immediately that
$\delta\eta = 0 = \eta\alpha$ and $0 \not\in \Toda \delta\eta\alpha$.

One can construct exact sequences in $\modR$ corresponding to
$\alpha$, $\eta$ and $\delta$ in a standard way by using injective
coresolutions and pull-backs:
$$
\begin{array}{rcl}
\alpha: &&
0 \to S \longrightarrow N \longrightarrow S \to 0,
\\
\eta: &&
0 \to N \longrightarrow R \longrightarrow N \longrightarrow
S \to 0,
\\
\delta: &&
0 \to N \longrightarrow S \oplus R \longrightarrow N \to 0.
\end{array}
$$

We leave the details for the reader. If we now splice these three
sequences, we get an exact sequence of legth six with the
required properties:
$$ 0 \to N \longrightarrow S \oplus R \longrightarrow R \longrightarrow N \longrightarrow N \longrightarrow S \to 0. $$

That is, this sequence satisfies the necessary condition
from~\cite{N} (the non-empty Toda bracket), but it still cannot be
obtained from the snake lemma because the Toda bracket does not
contain the zero morphism.
\end{expl}

% =============================================================================
\bibliographystyle{plain}
\bibliography{bibliography}{}

\end{document}